\documentclass[11pt]{article}%
\usepackage{amssymb}
\usepackage{eurosym}
\usepackage{amsfonts}
\usepackage{amsmath}
\usepackage{graphicx}
\usepackage{epstopdf}%
\providecommand{\U}[1]{\protect\rule{.1in}{.1in}}
\newtheorem{theorem}{Theorem}

\newtheorem{definition}[theorem]{Definition}

\newtheorem{lemma}[theorem]{Lemma}

\newtheorem{proposition}[theorem]{Proposition}

\newenvironment{proof}[1][Proof]{\noindent\textbf{#1.} }{\ \rule{0.5em}{0.5em}}
\begin{document}

\title{Lemniscates as Trajectories of Quadratic Differentials}
\author{Faouzi Thabet\\University of Gabes, Tunisia}
\maketitle

\begin{abstract}
In this note, we study polynomial and rational lemniscates as trajectories of
related quadratic differentials. Some classic results can be then proved easily...

\end{abstract}


\textit{Keywords and phrases: Quadratic differentials.
Lemniscates.Fingerprints.}

\section{\bigskip A quadratic differential}

\bigskip Given a rational function $r\left(  z\right)  =\frac{p\left(
z\right)  }{q\left(  z\right)  },$ where $p\left(  z\right)  $ and $q\left(
z\right)  $ are two co-prime complex polynomials, we consider the quadratic
differential on the Riemann sphere $\widehat{%
\mathbb{C}
}$ :%
\begin{equation}
\varpi_{r}\left(  z\right)  =-\left(  \frac{r^{\prime}\left(  z\right)
}{r\left(  z\right)  }\right)  ^{2}dz^{2}=-\left(  \frac{p^{\prime}\left(
z\right)  q\left(  z\right)  -p\left(  z\right)  q^{\prime}\left(  z\right)
}{p\left(  z\right)  q\left(  z\right)  }\right)  ^{2}dz^{2}. \label{qdiff}%
\end{equation}
\emph{Finite critical points} and \emph{infinite critical points }of
$\varpi_{r}$ are respectively its zero's and poles; all other points of
$\widehat{%
\mathbb{C}
}$ are called \emph{regular points} of $\varpi_{r}.$

It is obvious that the partial fraction decomposition of $\frac{r^{\prime
}\left(  z\right)  }{r\left(  z\right)  }$ is as follows :
\begin{equation}
\frac{r^{\prime}\left(  z\right)  }{r\left(  z\right)  }=\sum_{p\left(
a\right)  q\left(  a\right)  =0}\frac{m_{a}}{z-a},\label{lucas}%
\end{equation}
where $m_{a}\in%
\mathbb{Z}
^{\ast}$ is the multiplicity of the zero $a$ of $p\left(  z\right)  q\left(
z\right)  .$ We deduce that
\[
\ \varpi_{r}\left(  z\right)  =-\frac{m_{a}^{2}}{\left(  z-a\right)  ^{2}%
}\left(  1+\mathcal{O}(z-a)\right)  dz^{2},\quad z\rightarrow a.
\]
In other words, the zero's of $p$ and $q$ are poles of order $2$ of
$\varpi_{r}$ with negative residue.

If
\[
\deg\left(  p^{\prime}q-pq^{\prime}\right)  =\deg\left(  pq\right)  -1,
\]
(in particular, if $\deg\left(  p\right)  \neq\deg\left(  q\right)  $), then,
with the parametrization $u=1/z$, we get
\[
\ \varpi_{r}\left(  u\right)  =-\frac{\left(  \deg\left(  p\right)
-\deg\left(  q\right)  \right)  ^{2}}{u^{2}}\left(  1+\mathcal{O}(u)\right)
du^{2},\quad u\rightarrow0;
\]
thus, $\infty$ is another double pole\emph{ }of $\varpi_{r}$ with negative
residue. If
\[
\deg\left(  p^{\prime}q-pq^{\prime}\right)  <\deg\left(  pq\right)  -2,
\]
then $\infty$ is zero of $\varpi_{r}$ with multiplicity greater than $1.$ In
the case
\[
\deg\left(  p^{\prime}q-pq^{\prime}\right)  =\deg\left(  pq\right)  -2,
\]
$\infty$ is a regular point.

\emph{Horizontal trajectories} (or just trajectories) of the quadratic
differential $\varpi_{r}$ are the zero loci of the equation%
\[
\varpi_{r}\left(  z\right)  >0,
\]
or equivalently%
\begin{equation}
\Re\int^{z}\frac{r^{\prime}\left(  t\right)  }{r\left(  t\right)  }%
dt=\log\left\vert r\left(  z\right)  \right\vert =\text{\emph{const}}.
\label{eq traj}%
\end{equation}
If $z\left(  t\right)  ,t\in%
\mathbb{R}
$ is a horizontal trajectory, then the function
\[
t\longmapsto\Im\int_{0}^{t}\frac{r^{\prime}\left(  z\left(  u\right)  \right)
}{r\left(  z\left(  u\right)  \right)  }z^{\prime}\left(  u\right)
du=\arg\left(  r\left(  z\left(  t\right)  \right)  \right)  -\arg\left(
r\left(  z\left(  0\right)  \right)  \right)
\]
is monotone.

The \emph{vertical} (or, \emph{orthogonal}) trajectories are obtained by
replacing $\Im$ by $\Re$ in equation (\ref{eq traj}). The horizontal and
vertical trajectories of the quadratic differential $\varpi_{r}$ produce two
pairwise orthogonal foliations of the Riemann sphere $\widehat{%
\mathbb{C}
}$.

A trajectory passing through a critical point of $\varpi_{r}$ is called
\emph{critical trajectory}. In particular, if it starts and ends at a finite
critical point, it is called \emph{finite critical trajectory}, otherwise, we
call it an \emph{infinite critical trajectory}. \bigskip If two different
trajectories are not disjoint, then their intersection must be a zero of the
quadratic differential.

The closure of the set of finite and infinite critical trajectories is called
the \emph{critical graph} of $\varpi_{r},$ we denote it by $\Gamma_{r}.$

The local and global structures of the trajectories is well known (more
details about the theory of quadratic differentials can be found in
\cite{Strebel},\cite{jenkins}, or \cite{F.Thabet}), in particular :

\begin{itemize}
\item At any regular point, horizontal (resp. vertical) trajectories look
locally as simple analytic arcs passing through this point, and through every
regular point of $\varpi_{p}$ passes a uniquely determined horizontal (resp.
vertical) trajectory of $\varpi_{p};$ these horizontal and vertical
trajectories are locally orthogonal at this point.

\item From each zero with multiplicity $m$ of $\varpi_{r},$ there emanate
$m+2$ critical trajectories spacing under equal angle $2\pi/(m+2)$.

\item Any double pole has a neighborhood such that, all trajectories inside it
take a loop-shape encircling the pole or a radial form diverging to the pole,
respectively if the residue is negative or positive.

\item A trajectory in the large can be, either a closed curve not passing
through any critical point (\emph{closed trajectory}), or an arc connecting
two critical points, or an arc that has no limit along at least one of its
directions (\emph{recurrent trajectory}).
\end{itemize}

The set $\widehat{%
\mathbb{C}
}\setminus\Gamma_{r}$ consists of a finite number of domains called the
\emph{domain configurations} of $\varpi_{r}.$ For a general quadratic
differential on a $\widehat{%
\mathbb{C}
}$, there are five kind of domain configuration, see \cite[Theorem3.5]%
{jenkins}. Since all the infinite critical points of $\varpi_{r}$ are poles of
order $2$ with negative residues, then there are three possible domain configurations:

\begin{itemize}
\item the \emph{Circle domain} : It is swept by closed trajectories and
contains exactly one double pole. Its boundary is a closed critical
trajectory. For a suitably chosen real constant $c$ and some real number
$r>0,$ the function $z\longmapsto r\exp\left(  c\int^{z}\frac{p^{\prime
}\left(  t\right)  }{p\left(  t\right)  }dt\right)  $ is a conformal map from
the circle domain $D$ onto the unit circle; it extends continuously to the
boundary $\partial D,$ and sends the double pole to the origin.

\item the \emph{Ring domain}: It is swept by closed trajectories. Its boundary
consists of two connected components. For a suitably chosen real constant $c$
and some real numbers $0<r_{1}<r_{2},$ the function $z\longmapsto\exp\left(
c\int^{z}\frac{p^{\prime}\left(  t\right)  }{p\left(  t\right)  }dt\right)  $
is a conformal map from the circle domain $D$ onto the annulus $\left\{
z:r_{1}<\left\vert z\right\vert <r_{2}\right\}  $ and it extends continuously
to the boundary $\partial D.$

\item the \emph{Dense domain : }It is swept by recurrent critical trajectory
i.e., the interior of its closure is non-empty. Jenkins Three-pole Theorem
(see \cite[Theorem 15.2]{Strebel}) asserts that a quadratic differential on
the Riemann sphere with at most three poles cannot have recurrent
trajectories. In general, the non-existence of such trajectories is not
guaranteed, but here, following the idea of \emph{level function} of
Baryshnikov and Shapiro (see \cite{shapiro barish}), the quadratic
differential $\varpi_{r}$ excludes the dense domain, as we will see in
Proposition \ref{no recurrent}.
\end{itemize}

A very helpful tool that will be used in our investigation is the
Teichm\"{u}ller lemma (see \cite[Theorem 14.1]{Strebel}).

\begin{definition}
\bigskip A domain in $\widehat{%
\mathbb{C}
}$ bounded only by segments of horizontal and/or vertical trajectories of
$\varpi_{r}$ (and their endpoints) is called $\varpi_{r}$-polygon.
\end{definition}

\begin{lemma}
[Teichm\H{u}ller]\label{teich lemma} Let $\Omega$ be a $\varpi_{r}$-polygon,
and let $z_{j}$ be the critical points on the boundary $\partial\Omega$ of
$\Omega,$ and let $t_{j}$ be the corresponding interior angles with vertices
at $z_{j},$ respectively . Then%
\begin{equation}
\sum\left(  1-\dfrac{\left(  m_{j}+2\right)  t_{j}}{2\pi}\right)  =2+\sum
n_{i}, \label{Teich equality}%
\end{equation}
where $m_{j}$ are the multiplicities of $z_{j},$ and $n_{i}$ are the
multiplicities of critical points of $\varpi_{r}$ inside $\Omega.\bigskip$
\end{lemma}

\section{Lemniscates}

We use the notations of \cite{khavinson}. Let us denote $n=\deg r=\max\left(
\deg p,\deg q\right)  >0$. For $c>0,$ the set
\begin{equation}
\Gamma_{r,c}=\{z\in%
\mathbb{C}
:|r(z)|=c\} \label{lemniscate}%
\end{equation}
is called rational lemniscate of degree $n.$ For more details, see
\cite{Sheil-Small}. From the point of view of the theory of quadratic
differentials, each connected component of the lemniscate $\Gamma_{r,c}$
coincides with a horizontal trajectory of $\varpi_{r}=-\left(  \frac
{r^{\prime}\left(  z\right)  }{r\left(  z\right)  }\right)  ^{2}dz^{2},$ as we
have seen in equation (\ref{eq traj}). The lemniscate $\Gamma_{r,c}$ is
entirely determined by the knowledge of the critical graph $\Gamma_{r}$ (which
is the union of the lemniscates $\Gamma_{r,\left\vert r\left(  a\right)
\right\vert },$ for all zero's $a$ of $\varpi_{r}$) of the quadratic
differential of $\varpi_{r}.$ In particular, if we denote by $n_{z}$ and
$n_{p}$ respectively the number of zero's and poles $r\left(  z\right)  $ in
$\widehat{%
\mathbb{C}
},$ then, from the local behavior of the trajectories, we see that, for
$c\rightarrow0^{+}$, the lemniscate $\Gamma_{r,c}$ is formed by exactly
$n_{z}$ disjoint closed curves each of them encircles a zero of $r\left(
z\right)  $, while for $c\rightarrow+\infty,$ $\Gamma_{r,c}$ is formed by
exactly $n_{p}$ disjoint closed curves each of them encircles a pole of
$r\left(  z\right)  $. If $\deg\left(  p^{\prime}q-pq^{\prime}\right)
<\deg\left(  pq\right)  -2,$ then, $\infty$ is a zero of $\varpi_{r}$ of
multiplicity $m\geq2,$ and there are $m+2$ critical trajectories emerging from
$\infty$ dividing in a symmetric way the complement of some zero centred ball
into $m+2$ connected components. See Figure \ref{FIG1}. In the rest of this
note, we assume that $\infty$ is a double pole, i.e., $\deg\left(  p^{\prime
}q-pq^{\prime}\right)  =\deg\left(  pq\right)  -1.$ \begin{figure}[tbh]
\begin{minipage}[b]{0.4\linewidth}
\centering\includegraphics[scale=0.28]{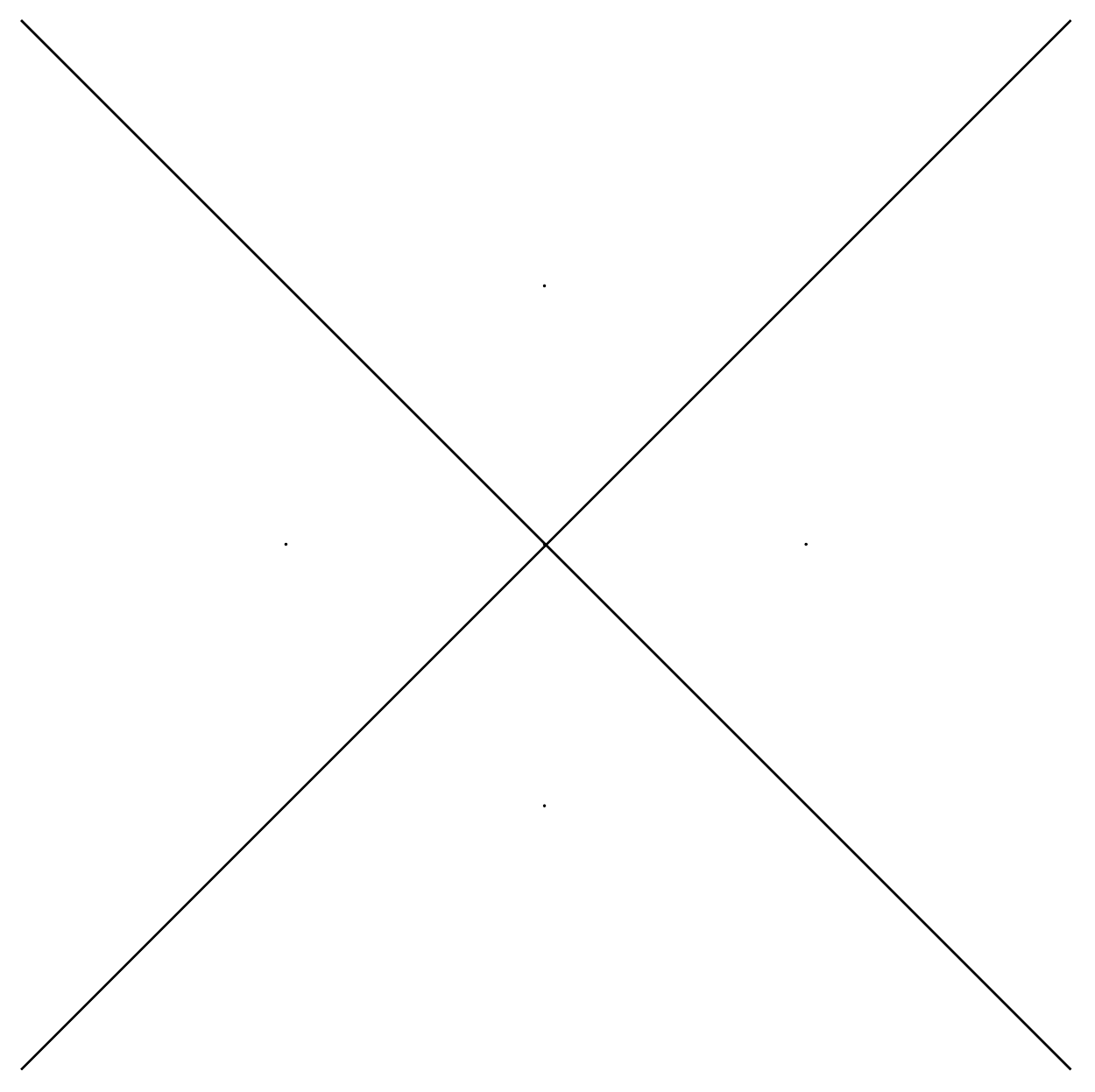}
\end{minipage}\hfill
\begin{minipage}[b]{0.4\linewidth} \includegraphics[scale=0.28]{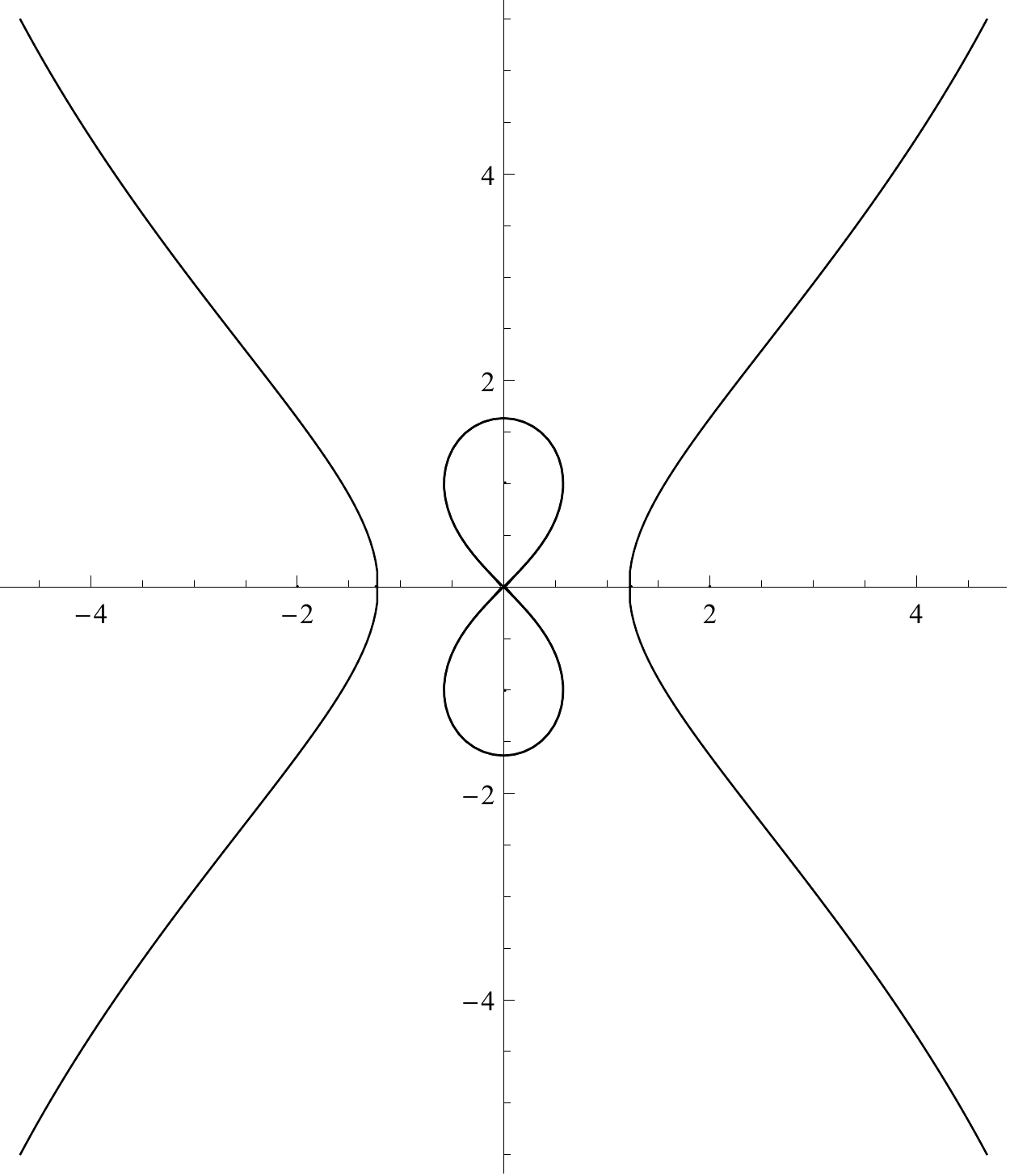}
\end{minipage}
\hfill\caption{Critical graphs of $\varpi_{r},$ $r=$ $\frac{x^{2}-1}{x^{2}+1}$
(left), and $r=$ $\frac{x^{2}-4}{x^{2}+1}$ (right).}%
\label{FIG1}%
\end{figure}

\begin{definition}
A quadratic differential on $\widehat{%
\mathbb{C}
}$ is called \emph{Strebel} if the complement to the union of its closed
trajectories has vanishing area.
\end{definition}

\begin{proposition}
\label{no recurrent}The quadratic differential $\varpi_{r}$ is Strebel.
\end{proposition}

\begin{proof}
Since the critical points of $\varpi_{r}$ are only zero's and double poles
with negative residues, it is sufficient to prove that $\varpi_{r}$ has no
recurrent trajectory. The function%
\[%
\begin{array}
[c]{cc}%
f:%
\mathbb{C}
\setminus\left\{  \text{poles of }r\left(  z\right)  \right\}  \longrightarrow%
\mathbb{R}%
& z\longmapsto\left\vert r\left(  z\right)  \right\vert
\end{array}
\]
is continuous, and constant on each horizontal trajectory of $\varpi_{r}.$ If
$\varpi_{r}$ has a recurrent trajectory, then, its domain configuration
contains a dense domain $D.$ Thus, the function $f$ must be constant on $D,$
which is clearly impossible by analyticity of the rational function
$z\longmapsto r\left(  z\right)  .$
\end{proof}

A necessary condition for the existence of a finite critical trajectory
connecting two finite critical points of $\varpi_{r}$ is the existence of a
Jordan arc $\gamma$ connecting them, such that
\begin{equation}
\Re\int_{\gamma}\frac{r^{\prime}\left(  t\right)  }{r\left(  t\right)  }dt=0.
\label{cond necess}%
\end{equation}
Unfortunately, this condition is not sufficient in general, as it can be shown
easily for the case of $r\left(  z\right)  =\left(  z^{2}-1\right)  \left(
z^{2}-4\right)  ;$ see Figure \ref{FIG2}. \begin{figure}[th]
\centering\includegraphics[height=1.8in,width=2.8in]{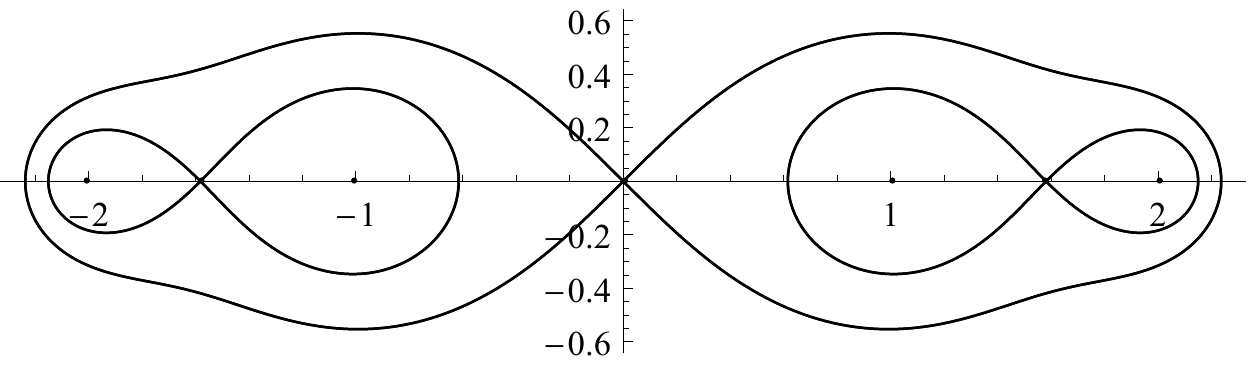}\caption{Critical
graph of $\varpi_{p}$, $p=\left(  z^{2}-1\right)  \left(  z^{2}-4\right)  .$}%
\label{FIG2}%
\end{figure}However, a more sufficient condition will be shown by the
following Proposition

\begin{proposition}
Let us denote $z_{1},...,z_{m}$ the finite critical points of $\varpi_{r}.$
If
\[
\left\vert w_{i}\right\vert =\left\vert w_{j}\right\vert =\max\left\{
\left\vert w_{k}\right\vert :=r\left(  z_{k}\right)  ;k=1,...,m\right\}
\]
for some $1$ $\leq i<j\leq m,$ then, there exists a finite critical trajectory
joining $z_{i}$ and $z_{j}.$ In particular, the critical graph $\Gamma_{r}$ is
connected, if and only if $\left\vert w_{1}\right\vert =\cdot\cdot
\cdot=\left\vert w_{m}\right\vert .$
\end{proposition}

\begin{proof}
If no finite critical trajectory joins $z_{i}$ and $z_{j},$ then a lemniscate
$\Gamma_{r,c},$ for some $c>\left\vert w_{i}\right\vert ,$ is not connected :
$\Gamma_{r,c}$ is a disjoint union of $s\geq2$ loops $L_{1},...,L_{s},$ each
of them encircles a part of the critical graph $\Gamma_{r}.$ Looking at each
of these loops as a $\varpi_{r}$-polygon and applying Lemma \ref{teich lemma},
we get :
\begin{equation}
0=2+\sum n_{k},k=1,...,s. \label{somme}%
\end{equation}
Making the sum of all equalities in (\ref{somme}), and taking into account our
assumption that $\left(  \deg\left(  p^{\prime}q-pq^{\prime}\right)
=\deg\left(  pq\right)  \right)  -1,$ we get%
\[
0=2s+2\left(  \deg\left(  p^{\prime}q-pq^{\prime}\right)  -\deg\left(
pq\right)  \right)  =2s-2;
\]
a contradiction. The second point is a mere consequence.

The numbers $w_{1}=r\left(  z_{1}\right)  ,...,w_{m}=r\left(  z_{m}\right)  $
are called the \emph{non-vanishing critical values} of $r\left(  z\right)  .$
\end{proof}

\section{\bigskip Fingerprints of polynomial lemniscates}

Here following a brief mention of the case of polynomial lemniscates
$\Gamma_{p,1}$. Let us denote by
\begin{align*}
\Omega_{-}  &  :=\{z\in%
\mathbb{C}
:|p(z)|<1\},\\
\Omega_{+}  &  :=\{z\in\widehat{%
\mathbb{C}
}:|p(z)|>1\}.
\end{align*}
The maximum modulus theorem asserts that $\Omega_{+}$ is a connected open
subset containing a neighborhood of $\infty$ in $\widehat{%
\mathbb{C}
}$.

\begin{definition}
A lemniscate $\Gamma_{p,1}$ of degree $n$ is \emph{proper }if it is smooth
($p^{\prime}\left(  z\right)  \neq0$ on $\Gamma_{p,1}$) and connected.
\end{definition}

\bigskip Let $z_{1,}...,z_{s},$ $s\leq n-1$ be the zero's (repeated according
to their multiplicity) of $\varpi_{p}.$ The non-vanishing critical values for
$p\left(  z\right)  $ are the values $w_{1}=p\left(  z_{1}\right)
,...,w_{s}=p\left(  z_{s}\right)  .$ For a smooth lemniscate $\Gamma_{p,1}$ of
degree $n$, the following characterizes the property of being proper through
the critical values :

\begin{proposition}
Assume that the lemniscate $\Gamma_{p,1}$ is smooth. Then, $\Gamma_{p,1}$ is
proper if and only if all the critical values $w_{1},...,w_{s}$ satisfy
$|w_{k}|<1.$
\end{proposition}

\begin{proof}
Proof of this Proposition can be found in \cite{khavinson}. We provide here a
more evident proof relying on quadratic differentials theory. The smoothness
of $\Gamma_{p,1}$ implies that it is not a critical trajectory. Suppose that
$|w_{k}|>1$ for some $k\in\left\{  1,...,s\right\}  ,$ and consider two
critical trajectories emerging from $z_{k}$ that form a loop $\gamma$. This
loop cannot intersect $\Gamma_{p,1},$ and $\gamma\cap\Omega_{-}\neq\emptyset$
since $\gamma$ contains a pole in its interior; a contradiction. The other
point is clear.
\end{proof}

Note that the interior $\Omega_{-}$ of a proper lemniscate of degree $n$ (or,
for a general smooth lemniscate, each component of ) is also simply connected,
since its complement is connected.

Let $\gamma$ be a $\mathcal{C}^{\infty}$ Jordan curve in $%
\mathbb{C}
;$ by a Jordan theorem, $\gamma$ splits $\widehat{%
\mathbb{C}
}$ into a bounded and an unbounded simply connected components $D_{-}$ and
$D_{+}.$ The Riemann mapping theorem asserts that there exist two conformal
maps $\phi_{-}:\Delta\longrightarrow$ $D_{-},$ and $\phi_{+}:\widehat{%
\mathbb{C}
}\setminus\overline{\Delta}\longrightarrow$ $D_{+},$ where $\Delta$ is the
unit disk. The map $\phi_{+}$ is uniquely determined by the normalization
$\phi_{+}\left(  \infty\right)  =\infty$ and $\phi_{+}\left(  \infty\right)
>0.$ It is well-known that $\phi_{-}$ and $\phi_{+}$ extend to $\mathcal{C}%
^{\infty}$-diffeomorphisms on the closure of their respective domain. The
\textit{fingerprint of }$\gamma$ is the map $k$ $:=$ $\phi_{+}^{-1}\circ
\phi_{-}:S^{1}\longrightarrow S^{1}$ from the unit circle $S^{1}$ to itself.
Note that $k$ is uniquely determined by up to post-composition with an
automorphism of $D$ onto itself. Moreover, the fingerprint $k$ is invariant
under translations and scalings of the curve $\gamma.$

\subsection{Lemniscates in a Circle Domain}

\bigskip

Let $a$ be a double pole of $\varpi_{p}$ ( $a=\infty$ or $p\left(  a\right)
=0$ ). Jenkins Theorem on the Configuration Domains of the quadratic
differential $\varpi_{p}$ asserts that there exists a connected neighborhood
$\mathcal{U}_{a}$ of $a$ (a Circle Domain of $\varpi_{p}$) bounded by finite
critical trajectories of $\varpi_{p},$ such that all trajectories of
$\varpi_{p}$ (lemniscates of $p$) inside $\mathcal{U}_{a}$ are closed smooth
curves encircling $a.$ Moreover, for a suitably chosen non-vanishing real
constant $c,$ the function
\[
\psi:z\longmapsto\exp\left(  c\int^{z}\frac{p^{\prime}\left(  t\right)
}{p\left(  t\right)  }dt\right)
\]
is a conformal map from $\mathcal{U}_{a}$ onto a certain disk centered in
$z=0.$ A more obvious form of it, is
\[
\psi\left(  z\right)  =\beta p\left(  z\right)  ^{c}%
\]
for some complex number $\beta.$ Baring in mind that $\psi$ is univalent near
$a$, we get
\[
c=\left\{
\begin{array}
[c]{c}%
\frac{1}{n},\text{ if }a=\infty\\
\frac{1}{\alpha},\text{ if }p\left(  a\right)  =0,\text{ }%
\end{array}
\right.
\]
where $\alpha$ is the multiplicity of $a$ if $p\left(  a\right)  =0.$ It
follows that the function
\[
z\longmapsto\left\{
\begin{array}
[c]{c}%
p\left(  z\right)  ^{\frac{1}{n}},\text{ if }a=\infty,\\
p\left(  z\right)  ^{\frac{1}{\alpha}},\text{ if }p\left(  a\right)  =0.
\end{array}
\right.
\]
is a conformal map from $\mathcal{U}_{a}$ onto a certain disk $\Delta_{a}$
centered in $z=0.$ We may assume for the sake of simplicity that $\Delta_{a}$
with a radius $R>1.$ For the given lemniscate $\Gamma_{p,1}$ in $\mathcal{U}%
_{a}$ (see Figure \ref{FIG3} ), it is straightforward that the function
$z\longmapsto p\left(  z\right)  ^{\frac{1}{\alpha}}$ maps $\Omega_{-}$
conformally onto the unit disk $\Delta.$ Thus,
\[
\left\{
\begin{array}
[c]{c}%
\phi_{+}^{-1}\left(  z\right)  =p\left(  z\right)  ^{\frac{1}{n}},\text{ if
}a=\infty,\\
\phi_{-}^{-1}\left(  z\right)  =p\left(  z\right)  ^{\frac{1}{\alpha}},\text{
if }p\left(  a\right)  =0.
\end{array}
\right.  .
\]
\begin{figure}[tbh]
\begin{minipage}[b]{0.5\linewidth}
\centering\includegraphics[scale=0.4]{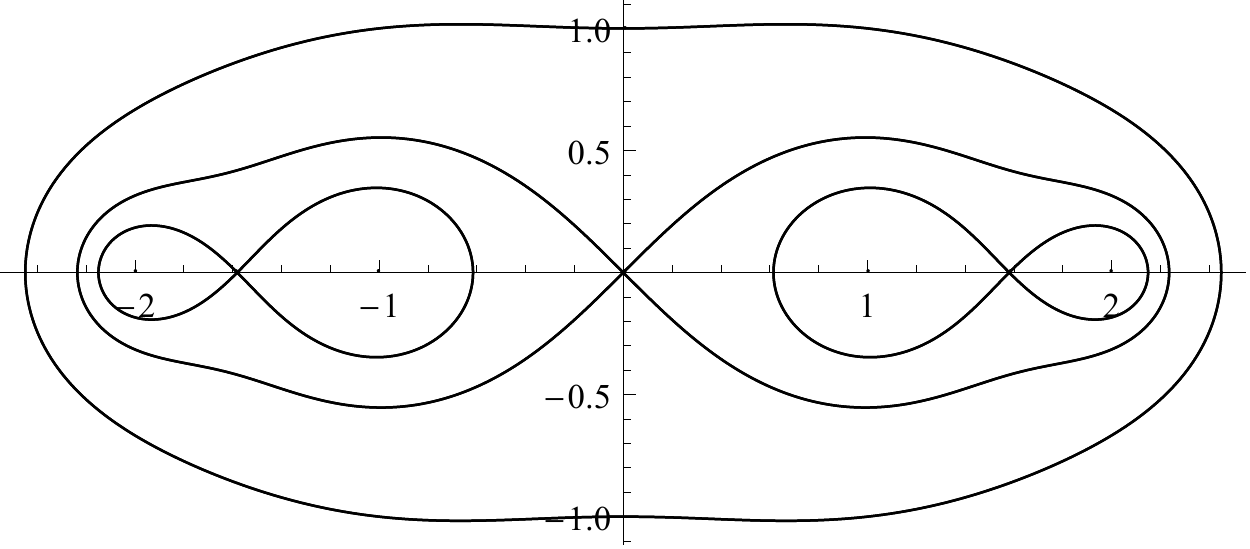}
\end{minipage}\hfill
\begin{minipage}[b]{0.5\linewidth} \includegraphics[scale=0.5]{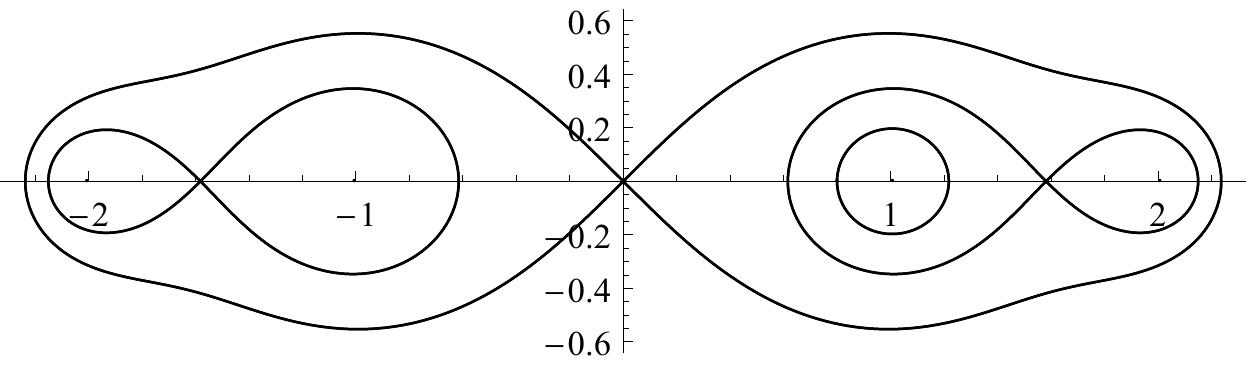}
\end{minipage}
\hfill\caption{Critical graph of $\varpi_{\left(  z^{2}-1\right)  \left(
z^{2}-4\right)  }$ and lemniscates in Circle Domains: $a=\infty$ (left), $a=1$
(right).}%
\label{FIG3}%
\end{figure}

\bigskip In the first case, we notice that $\Gamma_{p,1}$ is proper if and
only if $a=\infty;$ the next Theorem gives its fingerprint.

\begin{theorem}
[Ebenfelt, Khavinson and Shapiro ]The fingerprint $k$ $:S^{1}\longrightarrow
S^{1}$ of a proper lemniscate $\Gamma_{p,1}$ of the polynomial $p\left(
z\right)  =\prod_{k=1}^{n}\left(  z-\varsigma_{k}\right)  $ is given by
\[
k(z)=B(z)^{1/n},
\]
where $B$ is the Blaschke product of degree $n$%
\[
B(z)=e^{i\theta}\prod_{k=1}^{n}\frac{z-a_{k}}{1-\overline{a_{k}}z}%
\]
for some real number $\theta$, and $a_{k}=\phi_{-}\left(  \varsigma
_{k}\right)  ,k=1,,,n.$
\end{theorem}

In the case $p\left(  a\right)  =0$, let
\[
p\left(  z\right)  =\bigskip\left(  z-a\right)  ^{\alpha}p_{1}\left(
z\right)  ,\alpha\in%
\mathbb{N}
^{\ast};p_{1}\left(  z\right)  =\prod_{i=1}^{n-\alpha}\bigskip\left(
z-a_{i}\right)  ,p_{1}\left(  a\right)  \neq0.
\]
With the normalization $\phi_{+}\left(  z\right)  \rightarrow\infty$ as
$z\rightarrow\infty,$ the function
\[
z\longmapsto\ \frac{p\circ\phi_{+}\left(  z\right)  }{\prod_{i=1}^{n-\alpha
}\frac{z-\phi_{+}^{-1}\left(  a_{i}\right)  }{1-\overline{\phi_{+}^{-1}\left(
a_{i}\right)  }z}};\left\vert z\right\vert \geq1
\]
is holomorphic in $%
\mathbb{C}
\setminus\overline{\Delta}$, does not vanish there, is continuous in $%
\mathbb{C}
\setminus\Delta,$ and has modulus one on $\partial\Delta=S^{1}$. We deduce the
existence of $\theta\in%
\mathbb{R}
$ such that
\[
p\circ\phi_{+}\left(  z\right)  =e^{i\theta}z^{n}\prod_{i=1}^{n-\alpha}%
\frac{z-\phi_{+}^{-1}\left(  a_{i}\right)  }{1-\overline{\phi_{+}^{-1}\left(
a_{i}\right)  }z};\left\vert z\right\vert \geq1,
\]
which proves the

\begin{theorem}
Let $\Gamma_{p,1}$ be a smooth connected lemniscate such that $z=a$ is the
only zero of $p$ in $\Omega_{-}$ $.$ The fingerprint $k$ $:S^{1}%
\longrightarrow S^{1}$ of $\Gamma_{p,1}$ is given by
\[
k^{-1}(z)=z^{\frac{n}{\alpha}}B_{1}\left(  z\right)  ^{\frac{1}{\alpha}}.
\]
where $B_{1}\left(  z\right)  $ is the Blaschke product
\[
B_{1}\left(  z\right)  =e^{i\theta}\prod_{i=1}^{n-\alpha}\frac{z-\phi_{+}%
^{-1}\left(  a_{i}\right)  }{1-\overline{\phi_{+}^{-1}\left(  a_{i}\right)
}z}.
\]

\end{theorem}

\subsection{Lemniscates in a Ring Domain}

In the following, let $\mathcal{U}$ be a Ring Domain of the quadratic
differential $\varpi_{p}.$ It is bounded by two lemniscates $\Gamma_{p,r}$ and
$\Gamma_{p,R}.$ We may assume that
\[
0<r<1<R.
\]
For the sake of simplicity, we may assume that $p$ has exactly two different
zeros $a$ and $b$ in the bounded domain of $%
\mathbb{C}
$ with boundary $\Gamma_{p,r}.$
\begin{align*}
p\left(  z\right)   &  =\bigskip\left(  z-a\right)  ^{\alpha}\left(
z-b\right)  ^{\beta}p_{2}\left(  z\right)  ,\alpha,\beta\in%
\mathbb{N}
^{\ast};\\
p_{2}\left(  z\right)   &  =\prod_{i=1}^{n-\left(  \alpha+\beta\right)
}\bigskip\left(  z-a_{i}\right)  ,p_{2}\left(  a\right)  p_{2}\left(
b\right)  \neq0.
\end{align*}
We consider the lemniscate $\Gamma_{p,1}$ of $p$ in $\mathcal{U}$ (see Figure
\ref{FIG4} ). \begin{figure}[th]
\centering\includegraphics[height=1.8in,width=2.8in]{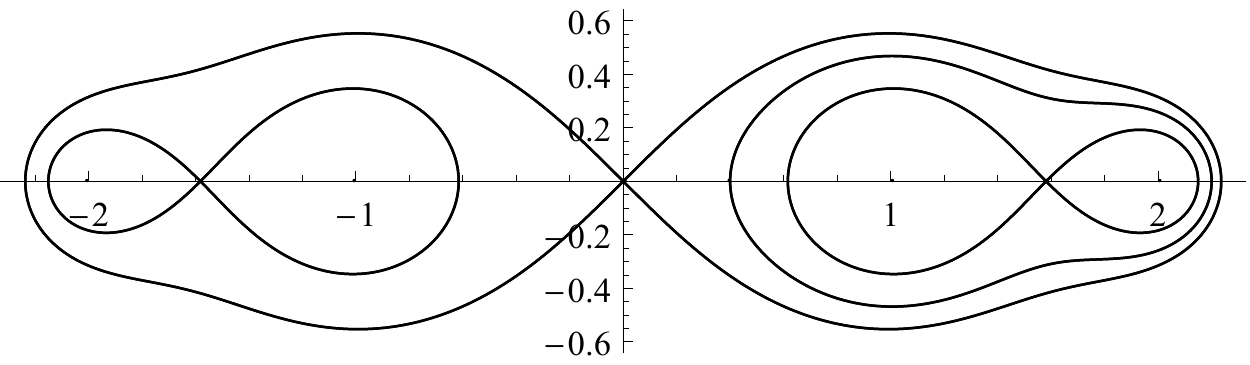}\caption{Critical
graph of $\varpi_{\left(  z^{2}-1\right)  \left(  z^{2}-4\right)  }$ with a
lemniscate in a Ring Domain ( $a=1,b=2$ ).}%
\label{FIG4}%
\end{figure}Since the function
\[
z\longmapsto p\circ\phi_{-}\left(  z\right)  =\bigskip\left(  \phi_{-}\left(
z\right)  -a\right)  ^{\alpha}\left(  \phi_{-}\left(  z\right)  -b\right)
^{\beta}p_{2}\left(  \phi_{-}\left(  z\right)  \right)
\]
is holomorphic in $\Delta$, is continuous in $\overline{\Delta},$ has
$\phi_{-}^{-1}\left(  a\right)  $ and $\phi_{-}^{-1}\left(  b\right)  $ as
unique zeros (with multiplicities $\alpha$ and $\beta$) in $\Delta,$ and has
modulus one on $\partial\Delta$. We deduce that there exists $\theta_{1}\in%
\mathbb{R}
$ such that$\bigskip$%
\[
p\circ\phi_{-}\left(  z\right)  =\bigskip e^{i\theta_{1}}\left(  \frac
{z-\phi_{-}^{-1}\left(  a\right)  }{1-\overline{\phi_{-}^{-1}\left(  a\right)
}z}\right)  ^{\alpha}\left(  \frac{z-\phi_{-}^{-1}\left(  b\right)
}{1-\overline{\phi_{-}^{-1}\left(  b\right)  }z}\right)  ^{\beta};\left\vert
z\right\vert \leq1.
\]
Reasoning like in the previous subsection on $\phi_{+}\left(  z\right)  ,$ we
get for some $\theta_{2}\in%
\mathbb{R}
$
\[
p\circ\phi_{+}\left(  z\right)  =\bigskip e^{i\theta_{2}}z^{n}\prod
_{i=1}^{n-\left(  \alpha+\beta\right)  }\frac{z-\phi_{+}^{-1}\left(
a_{i}\right)  }{1-\overline{\phi_{+}^{-1}\left(  a_{i}\right)  }z};\left\vert
z\right\vert \geq1.
\]
Combining the last two equalities for $\left\vert z\right\vert =1,$ we obtain
the following

\begin{theorem}
Let $\Gamma_{p,1}$ be a smooth connected lemniscate such that $\Omega_{-}$
contains exactly two different zeros $a$ and $b$ of $p$ with respective
multiplicities $\alpha$ and $\beta.$The fingerprint $k$ $:S^{1}\longrightarrow
S^{1}$ of $\Gamma_{p,1}$ satisfies the functional equation
\[
\left(  B\circ k\right)  \left(  z\right)  =A\left(  z\right)  ;\left\vert
z\right\vert =1.
\]
where $A$ and $B$ are the Blaschke products given by
\[
\bigskip B\left(  z\right)  =e^{i\theta}\left(  \frac{z-\phi_{-}^{-1}\left(
a\right)  }{1-\overline{\phi_{-}^{-1}\left(  a\right)  }z}\right)  ^{\alpha
}\left(  \frac{z-\phi_{-}^{-1}\left(  b\right)  }{1-\overline{\phi_{-}%
^{-1}\left(  b\right)  }z}\right)  ^{\beta},\theta\in%
\mathbb{R}
.
\]%
\[
A\left(  z\right)  =z^{n}B_{2}\left(  z\right)  =z^{n}\prod_{i=1}^{n-\left(
\alpha+\beta\right)  }\frac{z-\phi_{+}^{-1}\left(  a_{i}\right)  }%
{1-\overline{\phi_{+}^{-1}\left(  a_{i}\right)  }z}.
\]

\end{theorem}

\bigskip

\texttt{Institut Sup\'{e}rieur des Sciences }

\texttt{Appliqu\'{e}es et de Technologie de Gab\`{e}s, }

\texttt{Avenue Omar Ibn El Khattab, 6029. Tunisia.}

\texttt{E-mail adress:faouzithabet@yahoo.fr}
\end{document}